\documentclass[12pt]{amsart}
\usepackage{amssymb,amsmath,amscd,graphicx,fontenc,amsthm,mathrsfs,latexsym,enumerate}
\usepackage[pdftex,bookmarks,colorlinks,breaklinks]{hyperref} 
\hypersetup{linkcolor=blue,citecolor=red}
\newtheorem{theorem}{Theorem}
\newtheorem{corr}{Corollary}
\newtheorem{lemma}{Lemma}

\newtheorem{remark}{Remark}

\theoremstyle{definition}
\newcommand{\Gal}{\mathrm{Gal}}
\newcommand{\SL}{\mathrm{SL}}
\newcommand{\GL}{\mathrm{GL}}
\usepackage[top=3.2cm,bottom=3.2cm,right=2cm,left=2cm,twoside=false]{geometry}
\begin{document}
\title[The Density of a family of monogenic number fields]{The Density of a family of monogenic number fields}
\author[M. Bardestani]{Mohammad Bardestani}
\address{D\'{e}partement de math\'{e}matiques et de statistique
Universit\'{e} de Montr\'{e}al
CP 6128, succ. Centre-ville
Montr\'{e}al, Qu\'{e}bec H3C 3J7
Canada.
}
\address{
Current address:  Department of Mathematics and Statistics, University of Ottawa, 585 King Edward, Ottawa, ON K1N 6N5, Canada.}
\email{mbardest@uottawa.ca}
\date{}
\begin{abstract}
A monogenic polynomial $f$ is a monic irreducible polynomial with integer coefficients which produces a monogenic number field. For a given prime $q$, using the Chebotarev density theorem, we will show the density of primes $p$ such that $t^q-p$ is monogenic is bigger or equal than $(q-1)/q$. We will also prove that, when $q=3$, the density of primes $p$ which $\mathbb{Q}(\sqrt[3]{p})$ is non-monogenic is at least $1/9$.
\end{abstract}
\maketitle
\section{Introduction}
Let $K$ be an algebraic number field of degree $n$ and $\mathcal{O}_K$ denote its ring of integers. We call $K$ {\bf monogenic} if there exists an element $\alpha\in \mathcal{O}_K$  such that $\mathcal{O}_K=\mathbb{Z}[\alpha]$. It is a classical problem in algebraic number theory to identify if a number field $K$ is monogenic or not. Hasse (see~\cite{Gy}, page 193) asked if one could give an arithmetic characterization of monogenic number fields. The quadratic and cyclotomic number fields are monogenic, but this is not the case in general. Dedekind (see~\cite{Narkiewicz}, page 64) was the first who noticed this by giving an example of a cubic field generated by a root of $t^3-t^2-2t-8$. Let $f(t)\in\mathbb{Z} [t]$ be a monic irreducible polynomial. We call $f(t)$ {\bf monogenic} if $\mathbb{Q}[t]/(f(t))\cong\mathbb{Q}(\alpha)$ is monogenic, where $\alpha$ is a root of $f(t)$. In this paper, we will study the distribution of a family of monogenic polynomials.
\begin{theorem}\label{main}
Let $p$ and $q$ be prime numbers, where $q\geq 3$. Consider the polynomial $f_p(t):=t^q-p$.
Then, we have
$$
\liminf_{x\to\infty}\frac{1}{\pi(x)}\#\{p\leq x: f_p(t) \text{\, is monogenic}\}\geq\frac{q-1}{q},
$$
where $\pi(x)$ denotes the number of primes less than $x$.
\end{theorem}
The idea is to find a congruence condition on $p$ such that $f_p(t)=t^q-p$ is monogenic. This condition on $p$ reads as $p^{q-1}\not\equiv 1\pmod{q^2}$. Then we use the Chebotarev density theorem to count these primes.

When $q=3$, using a description of an integral basis for a pure cubic field and an explicit computation, we notice that the index form of $\mathbb{Q}(\sqrt[3]{hk^2})$ is represented by $hx^3-ky^3$ when $h^2\not\equiv k^2\pmod{9}$ and $(hx^3-ky^3)/9$ for $h^2\equiv k^2\pmod{9}$ (see Theorem~\ref{theoremded} and Lemma~\ref{disc}). Thus $\mathbb{Q}(\sqrt[3]{hk^2})$ being monogenic is equivalent to integral solubility of 
\begin{equation}\label{solu}
hx^3+ky^3=
\begin{cases}
 1 & \text{if $h^2\not\equiv k^2\pmod{9}$};\\
 9 & \text{if $h^2\equiv k^2\pmod{9}$}.
 \end{cases}
\end{equation}
In particular when $p$ is a prime, $\mathbb{Q}(\sqrt[3]{p})$ is monogenic for $p\equiv \pm 2,\pm 5\pmod{9}$. For $p\equiv \pm1\pmod{9}$ we obtain the following equation
\begin{equation}
px^3+y^3=9.
\end{equation}
By counting those primes $p\equiv \pm 1\pmod{9}$ where $9$ is not a cube in $\mathbb{F}_p$, we will find a lower bound for the density of non-monogenic cubic fields $\mathbb{Q}(\sqrt[3]{p})$. Notice that, when $p\equiv -1\pmod{9}$, $9$ is a cube in $\mathbb{F}_p$. Therefore we restrict ourselves to considering primes of the form $p\equiv 1 \pmod{9}$ and to computing the density of these primes where $9$ is not a cube modulo $p$. Let $K=\mathbb{Q}(\zeta_9,\sqrt[3]{9})$, where $\zeta_9$ is a primitive $9$th root of unity. By noticing that a prime $p$ splits completely in $K$ if and only if $p\equiv 1\pmod{9}$ and $9^{\frac{p-1}{3}}\equiv 1\pmod{p}$, and applying Chebotarev density theorem we will prove:
\begin{theorem}\label{nonmono}
The density of primes $p\equiv 1\pmod{9}$ such that the Thue's equation $px^3+y^3=9$,
 does not have any integral solution is at least $1/9$. This set of primes produces non-monogenic cubic fields $\mathbb{Q}(\sqrt[3]{p})$.
 \end{theorem}
We can also describe these primes by the following result:
\begin{theorem}\label{cubic} Let $p\equiv 1\pmod{9}$ be a prime number which can be represented by $7x^2+3xy+9y^2$. Then 
$\mathbb{Q}(\sqrt[3]{p})$ is non-monogenic. 
\end{theorem} 

\section{Monogenic Fields and Diophantine equations}
Generally speaking, one needs to solve a  Diophantine equation in order to show that a number field is monogenic.
It is useful to recall the following well-known statement (see~\cite{Neuk}, Proposition I.2.12).
\begin{lemma}\label{disc}
Let $K$ be a number field of degree $n$ and $\alpha_1,\dots, \alpha_n\in \mathcal{O}_K$ be linearly independent over $\mathbb{Q}$. Set $\mathcal{M}=\mathbb{Z}\alpha_1+\cdots+\mathbb{Z}\alpha_n$. Then
$Disc(\mathcal{M})=(\mathcal{O}_K:\mathcal{M})^2Disc(K)$. In particular,
$Disc(\alpha)=Ind(\alpha)^2Disc(K)$, if $\alpha\in \mathcal{O}_K$ and $K=\mathbb{Q}(\alpha)$, where $Ind(\alpha)=(\mathcal{O}_K:\mathbb{Z}[\alpha])$.
\end{lemma}
Choosing an integral basis for $K$ and writing $\alpha$ with respect to this integral basis, one can see that $Ind(\alpha)$ is a homogeneous form. In this section, we will focus on cubic fields. 

Let $K=\mathbb{Q}(\sqrt[3]{m})$, with $m\in\mathbb{Z}$ being a cube-free number, be a cubic field. We can assume $m=hk^2$ with $h,k>0$ and $hk$ is square-free. The following theorem is due to Dedekind (see~\cite{Williams} Theorem 7.3.2).
\begin{theorem}[Dedekind]\label{theoremded} Let $m$ and $K$ be as above, and let $\theta= m^{1/3}$. Then
\begin{enumerate}[\upshape (i)]
\item For $m^2\not\equiv 1\pmod{9}$, we have $Disc(K)=-27(hk)^2$, and the numbers $\left\{1,\theta,\theta^2/k\right\}$,
form an integral basis.
\item For $m\equiv \pm 1\pmod{9}$, we have $Disc(K)=-3(hk)^2$, and the numbers
$\left\{1,\theta,(k^2\pm k^2\theta+\theta^2)/3k\right\}$,
form an integral basis of $\mathcal{O}_K$.
\end{enumerate}
\end{theorem}
Notice that this Theorem shows $\mathbb{Q}(\sqrt[3]{p})$ is monogenic for primes $p\equiv \pm 2,\pm 5\pmod{9}$, which verifies Theorem~\ref{main} for $q=3$. For $p\equiv \pm1 \pmod{9}$, by invoking Theorem~\ref{theoremded} we obtain the following integral basis for $K=\mathbb{Q}(\sqrt[3]{p})$
\begin{equation*}
\left\{1,\theta,\frac{1\pm\theta+\theta^2}{3}\right\},
\end{equation*}
where $\theta=\sqrt[3]{p}$. Let
$
\alpha=a+b\theta+c\frac{1\pm\theta+\theta^2}{3}\in\mathcal{O}_K
$,
and assume $\alpha',\alpha''$ are the other conjugates of $\alpha$. It is easy to see that:
\begin{equation}
\begin{split}
\alpha-\alpha'&=(\theta-\theta')\left(\left(b \pm \frac{c}{3}\right) - \frac{c\theta''}{3} \right)\\
\alpha-\alpha''&=(\theta-\theta'')\left(\left(b \pm \frac{c}{3}\right) - \frac{c\theta'}{3} \right)\\
\alpha'-\alpha''&=(\theta'-\theta'')\left(\left(b \pm \frac{c}{3}\right) - \frac{c\theta}{3} \right)
\end{split}
\end{equation}
where $\theta'$ and $\theta''$ are the other conjugates of $\theta$. Therefore we have the following equalities: 
\begin{align*}
Disc(\alpha)&=Disc(\theta)\left(\left(b\pm\frac{c}{3}\right)^3-p\left(\frac{c}{3}\right)^3\right)^2
=-3^3p^2\left(\left(b\pm\frac{c}{3}\right)^3-p\left(\frac{c}{3}\right)^3\right)^2\\
&=-3p^2\left(3b^3\pm 3b^2c+bc^2+\frac{\pm1-p}{9}c^3\right)^2.
\end{align*}
Thus by applying Lemma~\ref{disc} and Theorem~\ref{theoremded} we have the following identity:
$$
Ind(\alpha)=\left|3b^3\pm 3b^2c+bc^2+\frac{\pm1-p}{9}c^3\right|.
$$
So to determine if $\mathbb{Q}(\sqrt[3]{p})$ is monogenic, for primes of the form $p\equiv\pm 1\pmod{9}$, we need to find the integral solutions of
\begin{equation}\label{p}
\left|3b^3\pm 3b^2c+bc^2+\frac{\pm1-p}{9}c^3\right|=1.
\end{equation}
Multiplying Equation~(\ref{p}) by 9 we obtain an equivalent equation
$|(3b\pm c)^3-pc^3|=9$,
which, for primes $p\equiv \pm 1\pmod{9}$, is equivalent to $px^3+y^3=9$.
Therefore, we obtain the following result
\begin{lemma}\label{2.3}
Let $p\equiv\pm 1\pmod{9}$ be a prime number. Then, $\mathbb{Q}(\sqrt[3]{p})$ is monogenic if and only if 
\begin{equation}\label{local}
px^3+y^3=9,
\end{equation}
has an integral solution.
\end{lemma}
\begin{remark}
Here, for simplicity, we found the index form of $\mathbb{Q}(\sqrt[3]{p})$. But the same computation gives us Equation~\eqref{solu}.
\end{remark}
Hence to construct a non-monogenic $\mathbb{Q}(\sqrt[3]{p})$, it would be enough to find a prime  $p\equiv \pm 1\pmod{9}$, such that Equation~\eqref{local} does not have any integral solution. One can find some of those primes by study the equation locally, for instance those primes $p$, such that $9$ is not a cube modulo $p$. Notice that $9$ is a cube if and only if $3$ is a cube in $\mathbb{F}_p$. Therefore we will briefly study the number of solutions of $h(t):=t^3-3$ in a finite field $\mathbb{F}_p$, denoted by $N_p(h(t))$, for all primes $p\geq 5$.
\begin{lemma}\label{sol} Let $N_p(h(t))$ be as above and $p\geq 5$ a prime number. Then
\begin{equation*}
N_p(h(t))=
\begin{cases}
1 & \text{if $p\equiv 2\pmod{3}$}\\
0 & \text{if $p=7x^2+3xy+9y^2$}\\
3 & \text{if $p=x^2+xy+61y^2$}
\end{cases}.
\end{equation*}
\end{lemma}
Let $E:=\mathbb{Q}[t]/(h(t))$ be the cubic field defined by $h(t)$ with the splitting field $L$, which contains the quadratic field $K:=\mathbb{Q}(\sqrt{-3})$. Let $\eta_1=\sqrt[3]{3},\eta_2,\eta_3$ be the conjugates of $\sqrt[3]{3}$ and define
$$\Delta:=\prod_{1\leq i<j\leq 3}(\eta_j-\eta_i)\in K.$$

For a prime $p\geq 5$, consider the Frobenius automorphism associated to $p$, say $\sigma_p\in \Gal(L/\mathbb{Q})$, which is unique up to conjugation. Regarding $\sigma_p$ as a permutation in $S_3$, we observe that 
$$\sigma_p(\Delta)=sgn(\sigma_p)\Delta.$$
Therefore $\sigma_p$ being even implies that $\sigma_p$ is a trivial element in $\Gal(K/\mathbb{Q})$, and thus $p$ splits completely in $K$. Also when $\sigma_p$ is an odd permutation we can conclude that $\sigma_p$ is not an identity element and therefore $p$ inerts in $K$. This shows    
$sgn(\sigma_p)=\left(\frac{p}{3}\right)$, where $\left(\frac{.}{3}\right)$ is denoted for the Legendre  symbol. Therefore $p\equiv 2\pmod{3}$ implies that $\sigma_p$ is a transposition and thus $h(t)=0$ has a unique solution in $\mathbb{F}_p$.  For $p\equiv 1\pmod{3}$, $\sigma_p$ is an even permutation, so $h(t)=0$ has either zero or three solutions in $\mathbb{F}_p$. Hence for $p\equiv 1\pmod{3}$, if $3$ is a cube in $\mathbb{F}_p$ we have $N_p(h(t))=3$, and if $3$ is not a cube we have $N_p(h(t))=0$.

One might find an alternative prove for this fact that $t^3-3$ has only one solution for $p\equiv 2\pmod{3}$, by looking at the homomorphism 
$$
\mathbb{F}_p^*\rightarrow\mathbb{F}_p^*\qquad a\mapsto a^3,
$$
and noticing that this an isomorphism. For $p\equiv 1\pmod{3}$, we have the 3rd root of unity in $\mathbb{F}_p$, and so one has either zero of three solutions. However the former method is more general, and can be applied for general polynomials (see~\cite{Serre}).

Using the cubic residue symbol, one can show that for primes $p\geq5$, $p$ can be presented by  $x^2 + xy + 61y^2$ if and only if $p\equiv 1\pmod{3}$ and $3$ is a cubic residue modulo $p$ (see~\cite{cox}, chapter 2, section 9).

Reduction theory of positive definite integral binary quadratic forms is easy to describe (for more details see~\cite{Flath,Venkov,cox}). For such a given form $f(x,y)=ax^2+bxy+cy^2$, by $\SL_2(\mathbb{Z})$ change of variable we can obtain a simpler form $f'(x,y)=a'x^2+b'xy+c'y^2$, where
$|b'|\leq a'\leq c'$  and in the case  $|b'|=a'$,  then  $b'=a'$;  and in the case  $a'=c'$,  then  $b'\geq 0$.
The discriminant of $x^2 + xy + 61y^2$ is $-243$ which has the class number $3$. More precisely, using the reduction algorithm explained briefly, there are, up to ${\SL}_2(\mathbb{Z})$ change of variables, three binary quadratic forms with discriminant $-243$. Namely $x^2 + xy + 61y^2, 7x^2\pm3xy+9y^2$, which are in the same genus. When $p\equiv 1\pmod{3}$, $p$ can be presented by only one of the form $x^2 + xy + 61y^2$ or $7x^2+3xy+9y^2$. Indeed we have the following:
\begin{lemma} Let $f_1,f_2$ be two integral binary quadratic forms of the same discriminant which represent the same prime. Then they are ${\GL}_2(\mathbb{Z})$-equivalent.
\end{lemma}
 \begin{proof}[Proof:]
Let $f(x,y)=ax^2+bxy+cy^2$ be an integral binary quadratic form, representing a prime $p$. We can assume $f(x,y)=px^2+bxy+cy^2$. Consider $\gamma=\begin{pmatrix}
   1 & 0  \\
   m & 1  \end{pmatrix} \in {\SL}_2 (\mathbb{Z})$. Then
$$
\gamma.f(x,y):=f((x,y)\gamma)=px^2+(2pm+b)xy+(c+bm+pm^2)y^2=px^2+b'xy+c'y^2.
$$
We can choose $m$ so that $-p<2pm+b\leq p$, so we have shown that any integral binary quadratic form that represents a prime $p$ is ${\SL}_2(\mathbb{Z})$-equivalent to $px^2+b'xy+c'y^2$, where $-p<b'\leq p$. Under $\GL_2(\mathbb{Z})$-equivalence, we can assume $0\leq b'\leq p$. This determines $b'$ (and hence $c'$) uniquely and finishes the proof.
\end{proof}
Lemma~\ref{sol} shows $px^3+y^3=9$ does not have any integral solutions for those primes $p\equiv 1\pmod{9}$ which can be represented by the quadratic form $7x^2+3xy+9y^2$, and hence Lemma~\ref{2.3} gives a proof for Theorem~\ref{cubic}.
\section{counting a family of non-monogenic cubic fields}
Let $\zeta_9$ denote a primitive $9$th root of unity. We will show the following result
\begin{lemma}\label{Thue}
A prime $p$ splits completely in $K=\mathbb{Q}(\zeta_9,\sqrt[3]{9})$ if and only if $p\equiv 1\pmod{9}$ and $9^{\frac{p-1}{3}}\equiv 1\pmod{p}$.
\end{lemma}
Let $\omega$ denote a primitive cube root of unity. To prove Lemma~\ref{Thue}, we need the following lemma, which is easy to prove.
\begin{lemma} Let $K=\mathbb{Q}(\zeta_9,\sqrt[3]{9})$. Then the following map is an isomorphism
$$
\psi:\Gal(K/\mathbb{Q})\rightarrow \frac{\mathbb{Z}}{3\mathbb{Z}}\rtimes\left(\frac{\mathbb{Z}}{9\mathbb{Z}}\right)^*\qquad
\sigma \mapsto (a(\sigma),b(\sigma)),
$$
where $\sigma(\sqrt[3]{9})=\omega^{a(\sigma)}\sqrt[3]{9}$ and $\sigma(\zeta_9)=\zeta_9^{b(\sigma)}$.
\end{lemma}
\begin{proof}[Proof of Lemma~\ref{Thue}] First let $p$ be an unramified prime in $K$ and $\sigma_p$ the Frobenius automorphism associated to $p$, which is unique up to conjugation. Then
$\zeta_9^{b(\sigma_p)}=\sigma_p(\zeta_9)\equiv \zeta_9^p\pmod{\mathfrak{p}}$,
where $\mathfrak{p}$ is a prime above $p$. Since $p$ is unramified we conclude 
\begin{equation}\label{ram1}
b(\sigma_p)\equiv p\pmod{9}.
\end{equation}
With the same reason we have
$\omega^{a(\sigma)}9^{\frac{1}{3}}=\sigma_p(9^{\frac{1}{3}})\equiv 9^{\frac{p}{3}}\pmod{\mathfrak{p}}$ which implies that 
\begin{equation}\label{ram2}
(\omega^{a(\sigma)}-9^{\frac{p-1}{3}})\in\mathfrak{p}.
\end{equation}
Now let $p\equiv 1\pmod{9}$ and $9^{\frac{p-1}{3}}\equiv 1\pmod{p}$. Notice that $p$ is unramified since $\gcd(p,3)=1$. Therefore by Equations~\eqref{ram1} and~\eqref{ram2} we have
\begin{enumerate}[\upshape (i)]
\item $b(\sigma_p)\equiv p\equiv 1\pmod{9}$ and so $b(\sigma)=1$.
\item $\omega^{a(\sigma)}-9^{\frac{p-1}{3}}\in\mathfrak{p}$ which implies $\omega^{a(\sigma)}-1\in\mathfrak{p}$ and so $a(\sigma)=0$.
\end{enumerate}
Thus $\psi(\sigma_p)=(0,1)$ and so $\sigma_p$ is the identity element. This means $p$ splits completely. Conversely, if $p$ splits completely then $a(\sigma)=0$ and $b(\sigma)=1$, which implies
\begin{enumerate}[\upshape (i)]
\item $\zeta_9\equiv\zeta_9^p\pmod{\mathfrak{p}}$ and so $p\equiv 1\pmod{9}$.
\item $9^{\frac{1}{3}}\equiv 9^{\frac{p}{3}}\pmod{\mathfrak{p}}$ and so $\left(9^{\frac{p-1}{3}}-1\right)\in\mathfrak{p}\cap\mathbb{Z}=(p)$.
\end{enumerate}
This finishes the proof.
\end{proof}
Since we will use the Chebotarev density theorem several times, let us recall it briefly. Let $K$ be a number field and assume $L/K$ is a Galois extension. To each prime ideal $\mathfrak{p}$ of $K$  unramified  in  $L$  there  corresponds  a  certain  conjugacy  class  $\mathcal{C}$  of $\Gal(L/K)$
consisting of the set of Frobenius automorphisms $\sigma$ attached to the prime ideals $\mathfrak{P}$ of $L$ which  lie over $\mathfrak{p}$.  Denote this conjugacy class by the Artin  symbol  $\left(\frac{L/K}{\mathfrak{p}}\right)$. For a
given conjugacy class  $\mathcal{C}$ of $\Gal(L/K)$,  let  $\pi_\mathcal{C}(x)$ denote the number of prime ideals $\mathfrak{p}$ of $K$ unramified in $L$ such that $\left(\frac{L/K}{\mathfrak{p}}\right)\in\mathcal{C}$ and $N_{L/K}(\mathfrak{p})\leq x$. By abuse of notation, the Frobenius automorphism is also represented by the Artin symbol. By the Chebotarev density theorem (see~\cite{Neuk}, Theorem 13.4, or~\cite{Lag, SerreI}) we have  
\begin{equation}
\lim_{x\to\infty}\frac{\pi_\mathcal{C}(x)}{\pi(x)}=\frac{|\mathcal{C}|}{[L:K]}.
\end{equation}
\begin{proof}[Proof of Theorem~\ref{nonmono}:]
Note that for $p\equiv 1\pmod{9}$, $9^{\frac{p-1}{3}}\equiv 1\pmod{p}$ is equivalent to $9$ being a cube in $\mathbb{F}_p$. Lemma~\ref{Thue} and the Chebotarev density theorem imply
$$
\frac{1}{\pi(x)}\#\{p\leq x: p\equiv 1\pmod{9}, 9 \,\,\text{is not a cube in}\,\, \mathbb{F}_p\}\xrightarrow[{x \to \infty }]{} \frac{1}{6}-\frac{1}{18}=\frac{1}{9}.
$$
\end{proof}
Since $t^3-9$ is an irreducible polynomial a famous conjecture due to Bunyakovsky says that there should be infinitely many primes of the form $t^3-9$ which are congruent to $\pm1$ modulo $9$. These primes produce monogenic fields. This shows the difficulty of characterizing monogenic fields even for pure cubic extensions. Monogenicity of cyclic cubic fields has been studied by  Dummit and Kisilevsky~\cite{kis}.
\section{Prime splitting and monogenic fields}
In a monogenic field $K$, the field discriminant is equal to the discriminant of the minimal polynomial of $\alpha$, where $\mathcal{O}_K=\mathbb{Z}[\alpha]$. Also, by using Dedekind's Theorem (see Theorem~\ref{prime}), it is easy to see how a prime splits, by looking at how the minimal polynomial of $\alpha$ splits modulo primes.
\begin{theorem}[Dedekind]\label{prime}
Let $K$ be a number field such that $\mathcal{O}_K=\mathbb{Z}[\alpha]$ for some $\alpha$ with the minimal polynomial $f(x)$. Let $\bar{f}(x)=\bar{p}_1(x)^{e_1}\cdots \bar{p}_g(x)^{e_g}$ be the factorization of the polynomial $\bar{f}(x)=f(x) \pmod{p}$ into irreducibles $\bar{p}_i(x)=p_i(x)\pmod{p}$, with all $p_i(x)\in\mathbb{Z}[x]$ monic. Then $\mathfrak{p}_i=p\mathcal{O}_K+p_i(\alpha)\mathcal{O}_K$ are the different prime ideals above $p$. Moreover $p\mathcal{O}_K=\mathfrak{p}_1^{e_1}\cdots\mathfrak{p}_g^{e_g}$.
\end{theorem}
\begin{proof}
See~\cite{Neuk} Proposition 8.3.
\end{proof}
Hence it is natural to see how prime splitting forces a number field to be non-monogenic. This idea was first noticed by Hensel. Indeed he constructed a family of $\mathbb{Z}/(3\mathbb{Z})$-extensions over $\mathbb{Q}$ such that $2$ splits completely, and since in $\mathbb{F}_2[t]$ there are only two linear polynomials, by invoking Theorem~\ref{prime}, he deduced that these fields are non-monogenic.

Hensel's idea can be extended easily to construct infinitely many non-monogenic Abelian number fields. Indeed let $l\equiv 1\pmod{n}$ be a prime and assume $n\geq 3$. Denote the unique $\mathbb{Z}/(n\mathbb{Z})$-subfield of $\mathbb{Q}(\zeta_l)$ by $K_n(l)$. The same method used to prove Lemma~\ref{cheb} shows that a prime $p$ splits completely in $K_n(l)$ if and only if $p\neq l$ and $t^n-p$ has a solution in $\mathbb{F}_l$. With these facts and Henel's idea, we can deduce that $K_n(l)$ is non-monogenic if $t^n-2$ has a solution in $\mathbb{F}_l$. Notice that different $l$ produces different $\mathbb{Z}/(n\mathbb{Z})$-extension since the discriminant of $K_n(l)$ is a function of $l$. Consider the Kummer extension $\mathbb{Q}(\zeta_n,\sqrt[n]{2})$. It is easy to show that $\Gal(\mathbb{Q}(\zeta_n,\sqrt[n]{2})/\mathbb{Q})$ can be embedded into the group $\left(\mathbb{Z}/(n\mathbb{Z})\right)\rtimes \left(\mathbb{Z}/(n\mathbb{Z})\right)^*$. Moreover one can show that a prime $l$ splits completely in $\mathbb{Q}(\zeta_n,\sqrt[n]{2})$ if and only if $l\equiv 1\pmod{n}$ and $t^n-2$ has a solution in $\mathbb{F}_l$. Then by the inequality $[\mathbb{Q}(\zeta_n,\sqrt[n]{2}):\mathbb{Q}]\leq n\varphi(n)$ and the Chebotarev density theorem we obtain:
\begin{align*}
\lim_{x\to\infty}\frac{1}{\pi(x)}\#\{l\leq x: l\equiv 1\pmod{n}, K_n(l) \text{\, is non-monogenic}\}\geq\frac{1}{n\varphi(n)}.
\end{align*}
We can extend this idea further. For a prime number $l\geq 3$ the field $K:=\mathbb{Q}(\zeta_{l^2})$ is a Galois extension with cyclic Galois group $\left(\mathbb{Z}/(l^2\mathbb{Z})\right)^*$. Let $\eta$ be a generator of this group and set $H:=\langle\eta^l\rangle$. Denote by $K_l$ the fixed field of $H$. Then $[K_l:\mathbb{Q}]=l$ and $\Gal(K/K_l)\cong H$.
\begin{lemma}\label{cheb}
$p$ splits completely in $K_l$ if and only if $p^{l-1}\equiv 1\pmod{l^2}$.
\end{lemma}
\begin{proof}[Proof:]
Assume that a prime $p$ split completely in $K_l$ and let $\sigma_p:=\left(\frac{K/\mathbb{Q}}{p}\right)$ be the Frobenius automorphism associated to $p$. Then $\left.\sigma_p\right|_{K_l}=\left.\left(\frac{K/\mathbb{Q}}{p}\right)\right|_{K_l}=\left(\frac{K_l/\mathbb{Q}}{p}\right)=id$ which implies $\sigma_p\in \Gal(K/K_l)\cong H$. Notice that $p\neq l$ is unramified in $K$. Thus  $\sigma_p(\zeta_{l^2})=\zeta_{l^2}^p$ since $\sigma_p(\zeta_{l^2})\equiv\zeta_{l^2}^p\pmod{\frak{p}}$, where $\frak{p}$ is a prime in $K$ above $p$. Under the canonical isomorphism $$\Gal(K/\mathbb{Q})\cong\left(\frac{\mathbb{Z}}{l^2\mathbb{Z}}\right)^*,$$ we observe that $p\in H=\langle\eta^l\rangle$. So for some integer $t$ we have
$p\equiv\eta^{tl}\pmod{l^2}$ that implies $p^{l-1}\equiv 1\pmod{l^2}$.

Conversely, let $p^{l-1}\equiv 1\pmod{l^2}$ and assume that $p\equiv\eta^t\pmod{l^2}$ for some integer $t$. Then
$p^{l-1}\equiv\eta^{t(l-1)}\equiv 1\pmod{l^2}$ that implies $l\mid t$.
Hence $p\in H$ and so $\sigma_p\in \Gal(K/K_l)$. Therefore $\left(\frac{K_l/\mathbb{Q}}{p}\right)=id$ which means $p$ splits completely in $K_l$.
\end{proof}
\begin{corr}
Let $l\geq 3$ be a prime number such that for some prime $p<l$ we have $p^{l-1}\equiv 1\pmod{l^2}$. Then $K_l$ is non-monogenic.
\end{corr}
As an application of Lemma~\ref{cheb} and the Chebotarev density theorem, one can calculate the density of $\#\{p\leq x: p^{l-1}\not\equiv 1\pmod{l^2}\}$ which will be used in the proof of Theorem~\ref{main}.
\begin{theorem}\label{ch} With the notations of Lemma~\ref{cheb} we have
$$\#\left\{p\leq x: p^{l-1}\not\equiv 1\pmod{l^2}\right\}=\frac{l-1}{l}\pi(x)(1+o(1)).$$
\end{theorem}
\begin{proof}[Proof:]
Note that $\mathcal{C}=\Gal(K_l/\mathbb{Q})-\{e\}$ is stable under conjugation, where $e$ is the identity element in the Galois group, and $\mathcal{C}$ corresponds to the set of non-split primes by Lemma~\ref{cheb}. Then the Chebotarev density theorem implies our theorem.
\end{proof}
\begin{remark}
Let $K/\mathbb{Q}$ be a cyclic extension of prime degree $l\geq 5$. Gras~\cite{Gr} in her beautiful paper, using a result of Leopoldt, showed that $K$ is non-monogenic  unless  $2l+1=p$ is a prime and $K=\mathbb{Q}(\zeta_p+\zeta_p^{-1})$.
\end{remark}
\section{Proof of Theorem~\ref{main}}
Recall that a polynomial $f(t)=t^n+a_{n-1}t^{n-1}+\cdots+a_1t+a_0$ is called an {\it Eisenstein polynomial at a prime $p$} when $p\mid a_i$ for all $0\leq i\leq n-1$ and $p^2\nmid a_0$.
Let $f(t)=t^n+a_{n-1}t^{n-1}+\cdots+a_1t+a_0$ be an Eisenstein polynomial at $p$, and let $K=\mathbb{Q}(\alpha)$ be a field generated by a root of $f(t)$. We will show that for any integers, $c_0, c_2,\dots, c_{n-1}$ we have the following formula
\begin{equation}\label{eq1}
N_{K/\mathbb{Q}}(c_0+c_1\alpha+\dots+c_{n-1}\alpha^{n-1})\equiv c_0^n \pmod{p},
\end{equation}
from which we deduce the following well-know result (see also~\cite{shaf}, section 2.5, Exercise 19).
\begin{lemma}\label{div}
Suppose $K = \mathbb{Q}(\alpha)$, where $\alpha$ is a root of an Eisenstein polynomial at $p$. Then we have $p\nmid [\mathcal{O}_K:\mathbb{Z}[\alpha]]$.
\end{lemma}
\begin{proof}[Proof:]
Suppose $p|[\mathcal{O}_K,\mathbb{Z}[\alpha]]$. Therefore there exists an algebraic integer $\theta\in\mathcal{O}_K\setminus \mathbb{Z}[\alpha]$, such that $p\theta\in\mathbb{Z}[\alpha]$. Hence for some integers $c_0,c_1,\dots, c_{n-1}$ we have $p\theta=c_0+c_1\alpha+\dots+c_{n-1}\alpha^{n-1}$,
and so
$$
p^nN_{K/\mathbb{Q}}(\theta)=N_{K/\mathbb{Q}}(c_0+c_1\alpha+\dots+c_{n-1}\alpha^{n-1})\equiv c_0^n\pmod{p},
$$
which implies $p\mid c_0$. Note that $p\mid N_{K/\mathbb{Q}}(\alpha)$ and $p^2\nmid N_{K/\mathbb{Q}}(\alpha)$. Therefore this process and Equation~\eqref{eq1} imply $p\mid c_i$ for all $i$, which is a contradiction. To complete the proof it remains to prove Equation~\eqref{eq1}. Let $E$ be the Galois closure of $K$, and assume $\mathfrak{p}$ is a prime in $E$ above $p$. Since $p\mid a_i$, then $\alpha_i\in\mathfrak{p}$, where $\alpha=\alpha_1,\alpha_2,\dots,\alpha_n$ are the conjugates of $\alpha$. Note that
$$
N_{K/\mathbb{Q}}(c_0+c_1\alpha+\dots+c_{n-1}\alpha^{n-1})=\prod_{i=1}^n(c_0+c_1\alpha_i+\dots+c_{n-1}\alpha_i^{n-1}),
$$
which implies that
$N_{K/\mathbb{Q}}(c_0+c_1\alpha+\dots+c_{n-1}\alpha^{n-1})-c_0^n\in \mathfrak{p}\cap \mathbb{Z}=(p)$.
\end{proof}
Lemma~\ref{div} will allow us to find an arithmetic condition on $p$ such that $f_p(t)$, defined in Theorem~\ref{main}, produces a monogenic field.
\begin{proof}[Proof of Theorem~\ref{main}:]
Let $K:=E_{f_p}=\mathbb{Q}(\alpha)$ be the field obtained by adjoining a root of $f_p(t)$ to $\mathbb{Q}$. Since $f_p(t)$ is an Eisenstein polynomial at $p$, we have that $p\nmid [\mathcal{O}_K:\mathbb{Z}[\alpha]]$.
It is easy to show that $|Disc(f_p)|=q^qp^{q-1}$. For $p^{q-1}\not\equiv 1 \pmod{q^2}$ we see that
$$
f_{p}(t+p)=(t+p)^q-p=t^q+\binom{q}{1}pt^{q-1}+\cdots+\binom{q}{q-1}p^{q-1}t+(p^q-p),
$$
from which we deduce $f_p(t+p)$ is an Eisenstein polynomial at the prime $q$ and so by Lemma~\ref{div} we obtain
$$
q\nmid [\mathcal{O}_K:\mathbb{Z}[\alpha-p]]=[\mathcal{O}_K:\mathbb{Z}[\alpha]]
.$$ 
Therefore $f_p(t)$ is monogenic by Lemma~\ref{disc}. Thus
$$
\#\{p\leq x: f_p(t) \text{\,is monogenic}\}\geq\#\{p\leq x: p^{q-1}\not\equiv 1 \pmod{q^2}\},
$$
which combined with Theorem~\ref{ch} proves our theorem.
\end{proof}
Theorem~\ref{main} can also be proven by Dirichlet's theorem on primes in arithmetic progressions. We thank Andrew Granville for pointing out this to the author. Indeed, for $1\leq i\leq q-1$, consider the following change of variable
\begin{equation*}
f_p(t+i)=(t+i)^q-p=t^q+\sum_{j=1}^{q-1}\binom{q}{j}t^ji^{q-j}+(i^q-p).
\end{equation*}
So to obtain an Eisenstein polynomial at $q$, we need to have the conditions $p\equiv i^q\equiv i\pmod{q}$
and $p\not\equiv i^q\pmod{q^2}$ that also imply $p^{q-1}\not\equiv 1\pmod{q^2}$. By the prime number theorem in arithmetic progressions, we get
\begin{align*}
&\lim_{x\to \infty}\frac{1}{\pi(x)}\{p\leq x: p\equiv i\pmod{q}, p\not\equiv i^q\pmod{q^2}\}=\\
&\lim_{x\to \infty}\frac{1}{\pi(x)}\{p\leq x: p\equiv i^q+qs\pmod{q^2}, 1\leq s\leq q-1\}=\frac{q-1}{q(q-1)}=\frac{1}{q}.
\end{align*}
This also proves Theorem~\ref{main}.

It should be mentioned that the simple change of variable $x+1$ also gives interesting examples.
Let $m=2^k$ be a power of $2$, and assume  $p\equiv 3\pmod{4}$ is a prime. Then $f(x)=x^m-p$ is an Eisenstein polynomial at $p$, with discriminant $-m^mp^{m-1}$. We now remark that $k\binom{m}{k}=m\binom{m-1}{k-1}$,
 which implies $f(x+1)=(x+1)^m-p=x^m+\sum_{j=1}^{m-1}\binom{m}{j}x^j+(1-p)$,
is an Eisenstein polynomial at $2$ and therefore
$
2\nmid [\mathcal{O}_K:\mathbb{Z}[\sqrt[m]{p}-1]]=[\mathcal{O}_K:\mathbb{Z}[\sqrt[m]{p}]]$.
 So $\mathbb{Q}(\sqrt[m]{p})$ is a monogenic number field.
\section{Some final remarks}
We can also fix a prime $p$ and vary $q$ in $t^q-p$. For example, when $p=2$, we want to know for which primes $q$, $t^q-2$ is monogenic. We should therefore study the distribution of primes $q$ such that $2^{q-1}\not\equiv 1 \pmod{q^2}$.
This is an interesting question, as it can be shown that  if  $2^{q-1}\not\equiv 1 \pmod{q^2}$, then the first case of Fermat's Last Theorem holds. Indeed, we expect that there are only few primes $q$ such that $2^{q-1}\equiv 1 \pmod{q^2}$. As far as I know, $1093$  and $3511$  are the  only primes  known to satisfy this  relation. For a number field $K$, let $\zeta_K(s)$ be the Dedekind zeta function of $K$ and assume its Laurent expansion at $s=1$ is
\begin{equation*}
\zeta_K(s)=c_{-1}(s-1)^{-1}+c_0+c_1(s-1)+\cdots\quad\quad (c_{-1}\neq 0).
\end{equation*}
Ihara~\cite{Ihara} in his interesting paper defined an analogue to the Euler-Kronecker constant
$\gamma_K:=c_0/c_{-1}$,
which is the same as the usual Euler constant for $K=\mathbb{Q}$. Let $K_q$ be the field we defined in Lemma~\ref{cheb} ($q=l$) and denote $\gamma_q:=\gamma_{K_q}$.
\begin{theorem}[Ihara~\cite{Ihara}] Assuming GRH,
if $\liminf \frac{\gamma_q}{q}=0$, then for each prime $p$, there are finitely many $q$ such that
$p^{q-1}\equiv 1\pmod{q^2}$.
\end{theorem}
Therefore by considering these assumptions we see,  for a fixed $p$, most of the time $t^q-p$ is monogenic.
These primes are called Wieferich primes. Granville and Soundararajan~\cite{Gran} in their paper related these primes to a conjecture of Erd{\H o}s asking if every positive integer is the sum of a square-free number and a power of $2$.

For a given prime $q\geq 3$, it would be interesting to classify the monogenicity of $K_p:=\mathbb{Q}(\zeta_q,\sqrt[q]{p})$ when $p\neq q$ varies. Chang~\cite{chang} considered this problem for $q=3$ and proved that $\mathbb{Q}(\sqrt[3]{2},\omega)$ is essentially the only monogenic field among the family $\mathbb{Q}(\sqrt[3]{p},\omega)$. However, it seems that for $q\geq 5$ the question is more delicate.
\section*{Acknowledgments}
I would like to thank Igor Shparlinski and Hershy Kisilevsky for the useful comments and Andrew Granville for his encouragements and advices.
\bibliographystyle{plain}
\bibliography{Mono}

\begin{thebibliography}{10}

\bibitem{Williams}
{\c{S}}aban Alaca and Kenneth~S. Williams.
\newblock {\em Introductory algebraic number theory}.
\newblock Cambridge University Press, Cambridge, 2004.

\bibitem{shaf}
A.~I. Borevich and I.~R. Shafarevich.
\newblock {\em Number theory}.
\newblock Translated from the Russian by Newcomb Greenleaf. Pure and Applied
  Mathematics, Vol. 20. Academic Press, New York-London, 1966.

\bibitem{chang}
Mu-Ling Chang.
\newblock Non-monogeneity in a family of sextic fields.
\newblock {\em J. Number Theory}, 97(2):252--268, 2002.

\bibitem{cox}
David~A. Cox.
\newblock {\em Primes of the form {$x^2 + ny^2$}}.
\newblock A Wiley-Interscience Publication. John Wiley \& Sons, Inc., New York,
  1989.
\newblock Fermat, class field theory and complex multiplication.

\bibitem{kis}
D.~S. Dummit and H.~Kisilevsky.
\newblock Indices in cyclic cubic fields.
\newblock In {\em Number theory and algebra}, pages 29--42. Academic Press, New
  York, 1977.

\bibitem{Flath}
Daniel~E. Flath.
\newblock {\em Introduction to number theory}.
\newblock A Wiley-Interscience Publication. John Wiley \& Sons, Inc., New York,
  1989.

\bibitem{Gran}
Andrew Granville and K.~Soundararajan.
\newblock A binary additive problem of {E}rd{\H o}s and the order of {$2\bmod
  p^2$}.
\newblock {\em Ramanujan J.}, 2(1-2):283--298, 1998.
\newblock Paul Erd{\H{o}}s (1913--1996).

\bibitem{Gr}
Marie-Nicole Gras.
\newblock Non monog\'en\'eit\'e de l'anneau des entiers des extensions
  cycliques de {${\bf Q}$} de degr\'e premier {$l\geq 5$}.
\newblock {\em J. Number Theory}, 23(3):347--353, 1986.

\bibitem{Gy}
K{\'a}lm{\'a}n Gy{\H{o}}ry.
\newblock Discriminant form and index form equations.
\newblock In {\em Algebraic number theory and {D}iophantine analysis ({G}raz,
  1998)}, pages 191--214. de Gruyter, Berlin, 2000.

\bibitem{Ihara}
Yasutaka Ihara.
\newblock On the {E}uler-{K}ronecker constants of global fields and primes with
  small norms.
\newblock In {\em Algebraic geometry and number theory}, volume 253 of {\em
  Progr. Math.}, pages 407--451. Birkh\"auser Boston, Boston, MA, 2006.

\bibitem{Lag}
J.~C. Lagarias and A.~M. Odlyzko.
\newblock Effective versions of the {C}hebotarev density theorem.
\newblock In {\em Algebraic number fields: {$L$}-functions and {G}alois
  properties ({P}roc. {S}ympos., {U}niv. {D}urham, {D}urham, 1975)}, pages
  409--464. Academic Press, London, 1977.

\bibitem{Narkiewicz}
W{\l}adys{\l}aw Narkiewicz.
\newblock {\em Elementary and analytic theory of algebraic numbers}.
\newblock Springer Monographs in Mathematics. Springer-Verlag, Berlin, third
  edition, 2004.

\bibitem{Neuk}
J{\"u}rgen Neukirch.
\newblock {\em Algebraic number theory}, volume 322 of {\em Grundlehren der
  Mathematischen Wissenschaften [Fundamental Principles of Mathematical
  Sciences]}.
\newblock Springer-Verlag, Berlin, 1999.
\newblock Translated from the 1992 German original and with a note by Norbert
  Schappacher, With a foreword by G. Harder.

\bibitem{SerreI}
Jean-Pierre Serre.
\newblock Quelques applications du th\'eor\`eme de densit\'e de {C}hebotarev.
\newblock {\em Inst. Hautes \'Etudes Sci. Publ. Math.}, (54):323--401, 1981.

\bibitem{Serre}
Jean-Pierre Serre.
\newblock On a theorem of {J}ordan.
\newblock {\em Bull. Amer. Math. Soc. (N.S.)}, 40(4):429--440 (electronic),
  2003.

\bibitem{Venkov}
B.~A. Venkov.
\newblock {\em Elementary number theory}.
\newblock Translated from the Russian and edited by Helen Alderson.
  Wolters-Noordhoff Publishing, Groningen, 1970.

\end{thebibliography}
\end{document}